\documentclass{article}
\usepackage[english]{babel}
\usepackage{amsmath}
\usepackage{amsthm}
\usepackage{amssymb}
\usepackage{amscd}
\usepackage[latin1]{inputenc}

\newtheorem{teor}{Theorem}[section]

\newtheorem{lema}[teor]{Lemma}
\newtheorem{prop}[teor]{Proposition}

\def\Im {\mathop{\rm Im}\nolimits}

\begin{document}

\title{Representation dimension of  extensions of hereditary algebras}

\author{Manuel Saorín\\ Departamento de Matemáticas\\
Universidad de Murcia, Aptdo. 4021\\
30100 Espinardo, Murcia\\
SPAIN\\ {\it msaorinc@um.es}}

\date{}

\thanks{The author thanks Idun Reiten, Oeyvind Solberg and
specially Steffen Oppermann for very fruitful discussions and
suggestions that helped to improve a preliminary version of this
work}

\thanks{The author is supported by research projects from
the D.G.I. of the Spanish Ministry of Education and the Fundación
"Séneca" of Murcia (exp. 04555/GERM/06), with a part of FEDER
funds}

\maketitle

\begin{abstract}

{\bf We show that if $H$ is a hereditary finite dimensional
algebra, $M$ is a finitely generated $H$-module and $B$ is a
semisimple subalgebra of $End_H(M)^{op}$, then the representation
dimension of $\Lambda =\begin{pmatrix} B & 0\\ M & H\end{pmatrix}$
is less or equal to $3$ whenever one of the following conditions
hold: i) $H$ is of finite representation type; ii) $H$ is tame and
$M$ is a direct sum of regular and preprojective modules; iii) $M$
has no self-extensions
 }
\end{abstract}

\vspace*{0.5cm}

The representation dimension of an Artin algebra is the infimum of
the global dimensions of the endomorphism algebras of
generator-cogenerators of its category of finitely generated
modules. It was introduced by Auslander (cf. \cite{A}) and, his
own words, it was aimed at being a measure of how far an (in this
paper always Artin) algebra is from being of finite representation
type. Indeed, in that same paper, Auslander proved that an algebra
is of finite representation type if, and only if, its
representation dimension is less or equal than $2$. While the
concept was essentially forgotten for almost thirty years, two
breaking recent results have put it into the spotlight again. On
one side,  Iyama \cite{I}  proved that the representation
dimension of an algebra is always finite, and, on the other,
Rouquier \cite{R} showed that all natural numbers can be attained.
With these two results at hand,  Artin algebras can be, at least
in theory, classified numerically. In addition, representation
dimension is invariant under stable equivalences (cf. \cite{G} and
\cite{D}) and, when restricted to self-injective algebras,
invariant under derived equivalences (cf. \cite{X}), facts that
allow to construct classes of algebras of a given representation
dimension from others having the same property.

It is a natural goal to discover classes of algebras of infinite
representation type that, from the point of view of representation
dimension, are  the nearest  to being of finite type, namely,
those having representation dimension equal to $3$. Examples of
these algebras available in the literature include the hereditary
\cite{A}, stably hereditary   \cite{X}, special biserial
\cite{EHIS}, Schur algebras of tame representation type \cite{H1},
local algebras of quaternion type \cite{H2}, selfinjective
algebras (socle equivalent to) weakly symmetric algebras of
Euclidean type \cite{BHS}, tilted and laura algebras \cite{APT}
and canonical algebras \cite{O1}.

In this paper we consider generalizations of one-point extensions
of hereditary algebras, namely, triangular algebras of the form $\Lambda =\begin{pmatrix} B & 0\\
M & H\end{pmatrix}$, where $H$ is a hereditary algebra over an
algebraically closed field $K$, $M$ is a (left) $H$-module and $B$
is a semisimple subalgebra of $End_H(M)^{op}$. We  find sufficient
conditions for those algebras to have representation dimension
$3$. Note that, due to recent results of Oppermann \cite{O}, every
wild algebra admits one-point extensions of representation
dimension $\geq 4$. So there are choices of $H$ and $M$ for which
$rep.dim(\Lambda )>3$.

The first main result of the paper, Proposition \ref{tame case},
states that $rep.dim(\Lambda )\leq 3$ whenever one of the
following two conditions holds: i) $H$ is of finite representation
type; ii) $H$ is tame and $M$ is a direct sum of preprojective and
regular modules. The second main result, Theorem \ref{main
theorem}, states that if $H$ is of infinite representation type
and $M$ has no self-extensions, then $rep.dim(\Lambda )=3$. The
proof of this theorem is based on the construction of an Auslander
generator $\hat{G}$ of $\Lambda-mod$ derived from the existence of
an Auslander generator $G$ of $H-mod$ which contains $M$ as a
direct summand (see Proposition \ref{Auslander generator of H
versus Lambda} and Proposition \ref{Auslander generator of H}).

With notation as above,  notice that if $B=B_1\times...\times B_r$
is the decomposition of $B$ into a direct product of simple
algebras, then the central idempotents of $B$ corresponding to
that decomposition give a decomposition $M=\oplus_{1\leq i\leq
r}M_i$ such that $B_i\subseteq End_H(M_i)$, for every $i=1,...,r$.
Moreover,  if $B_i\cong M_{n_i}(K)$ then $M_i\cong
\tilde{M}_i^{n_i}$ for some $H$-module $\tilde{M}_i$. It is clear
that $\Lambda$ is a basic algebra if, and only if, $H$ is basic
and $n_i=1$ for $i=1,...,r$. Without loss of generality, we can
and shall assume in the sequel that these two conditions hold.
Notice that, even with that restriction, the $M_i$ need not be
indecomposable. Notice also that if the chosen decomposition of
$M$ is the trivial one (i.e. $r=1$ and $M=M_1=\tilde{M}_1$ above),
then $\Lambda$ is just the
one-point extension $\begin{pmatrix} K & 0\\
M & H\end{pmatrix}$.

It is well-known (cf. [ARS]) that every $\Lambda$-module is then
identified by a triple $(V,X,f)$ consisting of a  $B$-module $V$,
an $H$-module $X$ and a homomorphism of $H$-modules
$f:M\otimes_BV\longrightarrow X$. Implicitly assuming $f$, we
shall
write $\Lambda$-modules as 2-entry columns $\begin{pmatrix} V\\
X\end{pmatrix}$ and the multiplications by elements of $\Lambda$
will be just left matrix multiplication. In that case the full
subcategory of $\Lambda -mod$ formed by the objects of the form $\begin{pmatrix} 0\\
X\end{pmatrix}$ is canonically identified with the category $H
-mod$.

 Note that if
$f^t:V\longrightarrow Hom_H(M,X)$ denotes the transpose of $f$,
which is a homomorphism of $B$-modules,  then due to the
semisimplicity of $B$ we have a decomposition
$\begin{pmatrix} V\\
X\end{pmatrix}\cong \begin{pmatrix} Ker(f^t)\\
0\end{pmatrix}\oplus\begin{pmatrix} Im(f^t)\\
X\end{pmatrix}$. That will allows us to reduce many arguments to the
case in which $V\subseteq Hom_H(M,X)$ is a $B$-submodule and the map
$M\otimes_BV\longrightarrow X$ is the canonical one: $m\otimes
v\rightsquigarrow v(m)$.

\begin{prop} \label{tame case}
Let $\Lambda =\begin{pmatrix}B & 0\\ M & H
\end{pmatrix}$ be as  above. If $H$ is of finite representation type or if
$H$ is tame and $M$ is a direct sum of regular and preprojective
$H$-modules , then there only finitely many indecomposable
torsionless $\Lambda$-modules up to isomorphism. In particular

\begin{center}
$rep.dim(\Lambda )\leq 3$
\end{center}
\end{prop}
\begin{proof}
The final assertion is a consequence of the first due to a recent
result of Ringel (\cite{R2}). As for the first sentence, notice that
the indecomposable projective $\Lambda$-modules are the projective
$H$-modules plus the modules $\begin{pmatrix} Kp_i\\
M_i
\end{pmatrix}$, where $p_i:M\longrightarrow M_i$ is the $i$-th projection
associated to the given decomposition of $M$. Since the radical of
$\begin{pmatrix}Kp_i\\ M_i
\end{pmatrix}$ is $\begin{pmatrix}0\\ M_i
\end{pmatrix}$ the indecomposable torsionless $\Lambda$-modules are the
projective ones   plus all the indecomposable $H$-modules in
$Sub(M)$. Therefore the case in which $H$ is of finite
representation type is obvious.

We assume in the sequel that $H$ is tame and $M$ admits a
decomposicion $M=X\oplus R$  as a direct sum of a preprojective
$H$-module $X$ and a regular
$H$-module $R$. If $\begin{pmatrix} u\\
v\end{pmatrix}:Z\rightarrowtail X^n\oplus R^n=M^n$ is an
monomorphism from the indecomposable $H$-module $Z$, then either
$u\neq 0$, in which case $Z$ is a (preprojective) predecessor of
some of the indecomposable summands of $X$ or, else, $v$ is a
monomorphism so that $Z\in Sub(R)$. By the well-known structure of
the subcategory of regular $H$-modules, the number of regular
indecomposable $H$-modules in $Sub(R)$ is finite. So the problem is
reduced to prove that if $R$ is any regular $H$-module, then
$Sub(R)$ contains only finitely many preprojective indecomposable
$H$-module. For that there is no loss of generality in assuming that
$R$ is multiplicity-free and, by adding some regular indecomposable
summands if necessary, also that $\tau_HR=R$. Notice that if
$f:Z\rightarrowtail R^m$ is an (indecomposable) monomorphism, where
$Z$ is a preprojective nonprojective indecomposable, then
$\tau_H(f):\tau_HZ\rightarrowtail\tau_H(R)^m\cong R^m$ is also an
(indecomposable) monomorphism (cf. [Kerner, Lemma 2.2]). In
particular, given any preprojective indecomposable $H$-module $Z$,
the set of natural numbers $S_Z=:\{n\geq 0:$ $\tau^{-n}Z\in
Sub(R)\}$ is closed under predecessors (i.e. $n\in S_Z$  implies
$n-1\in S_Z$). If there were infinitely many indecomposable
preprojective modules $Z$ in $Sub(R)$ we would conclude that there
is a projective indecomposable $H$-module $P$ such  that
$\tau^{-n}P\in Sub(R)$ for all $n\geq 0$.

Let us assume that such a $P$ exists. We then denote by $\varphi
(n)$ the largest of the positive integers $r$ such that there is a
monomorphism $\tau^{-n}P\rightarrowtail R^r$ which is an
indecomposable map. The argument in the above paragraph shows that
$\varphi (n-1)\geq\varphi (n)$. As a consequence the map $\varphi
:\mathbf{N}\longrightarrow\mathbf{N}$ is eventually constant, so
that we have a natural number $q$ such that $\varphi (n)=q$ for
$n>>0$. But then $\underline{dim}(\tau^{-n}P)\leq q\cdot
\underline{dim}(R)$ for all $n>>0$. That implies that there are only
finitely many dimension vectors of  modules in the $\tau$-orbit of
$P$. This is known to be false for preprojective indecomposable
modules are identified by their dimension vectors (cf. \cite{R1}).
\end{proof}

We now give an auxiliary result which is valid for every hereditary
algebra $H$.

\begin{prop} \label{Auslander generator of H versus Lambda}
Suppose that, in our situation, the $H$-module $M$ has no
selfextensions and that we have found an Auslander generator of $H
-mod$ containing $M$ as a direct summand.
Then the $\Lambda$-module $\hat{G}=\begin{pmatrix}0\\
G
\end{pmatrix}\oplus\begin{pmatrix} B\\ M
\end{pmatrix}\oplus D\Lambda$ satisfies that

\begin{center}
$gl.dim(End_\Lambda (\hat{G}))\leq 3$.
\end{center}
\end{prop}
\begin{proof}
Since the simple modules $\begin{pmatrix}Kp_i\\0 \end{pmatrix}$
are injective, whence belong to $Add(\hat{G})$,
without loss of generality, we can deal only with $\Lambda$-modules of the form  $\begin{pmatrix}V\\
X
\end{pmatrix}$, with $V$ a $B$-submodule of $Hom_H(M,X)$. In that
case, we claim that if $\begin{pmatrix}V\\ X
\end{pmatrix}$ is indecomposable and the canonical map
$f:M\otimes_BV\longrightarrow X$ is surjective, then
$[\begin{pmatrix}V\\ X
\end{pmatrix},\begin{pmatrix}0\\ G
\end{pmatrix}]=0$ and $rad[\begin{pmatrix}V\\ X
\end{pmatrix},\begin{pmatrix}B\\ M
\end{pmatrix}]=0$, where $[-,-]$ denotes $Hom_\Lambda (-,-)$. Indeed
in the first case a morphism $\begin{pmatrix}V\\ X
\end{pmatrix}\longrightarrow\begin{pmatrix}0\\ G
\end{pmatrix}$ is identified by a morphism $u:X\longrightarrow G$
such that $u\circ v=0$, for all $v\in V$. But then $u\circ f=0$ and
so $u=0$. In the second case we consider the initial decomposition
$M=M_1\oplus ...\oplus M_r$ and the associated projections
$p_i:M\longrightarrow M_i$, so that $\begin{pmatrix}B\\ M
\end{pmatrix}=\oplus_{1\leq i\leq r}\begin{pmatrix}Kp_i\\ M_i
\end{pmatrix}$ is the decomposition of $\begin{pmatrix}B\\ M
\end{pmatrix}$ into a direct sum of (projective) indecomposable $\Lambda$-modules.
By the above argument, a nonzero morphism $\psi :\begin{pmatrix}V\\
X
\end{pmatrix}\longrightarrow \begin{pmatrix}Kp_i\\ M_i
\end{pmatrix}$ cannot have image contained
in $rad(\begin{pmatrix}Kp_i\\ M_i
\end{pmatrix})=\begin{pmatrix}0\\ M_i
\end{pmatrix}$ because $\begin{pmatrix}0\\ M_i
\end{pmatrix}$ is a direct summand of $\begin{pmatrix}0\\ G
\end{pmatrix}$. Therefore $\psi$ is a (split) epimorphism.

We need to prove that the projective dimension of
$[\begin{pmatrix}V\\ X
\end{pmatrix},\hat{G}]$ as an $End_\Lambda (\hat{G})$ is $\leq 1$,
for all indecomposable $\Lambda$-modules $\begin{pmatrix}V\\ X
\end{pmatrix}$. By the above paragraph, if $f:M\otimes_BV\longrightarrow
X$ is surjective and $\begin{pmatrix}V\\ X
\end{pmatrix}\not\in Add(\begin{pmatrix}B\\ M
\end{pmatrix})$, then the only indecomposable summands of $\hat{G}$ on which
$[\begin{pmatrix}V\\ X
\end{pmatrix},-]$ does not vanish are the injectives. Since $X\in Fac(M)$ and $M$ has no self-extensions, we get that
$Ext_H^1(M,X)=0$ and then the minimal injective resolution of
$\begin{pmatrix}V\\ X
\end{pmatrix}$ is of the form

\begin{center}
$0\rightarrow \begin{pmatrix}V\\ X
\end{pmatrix}\hookrightarrow \begin{pmatrix}Hom_H(M,E(X))\\ E(X)
\end{pmatrix}\longrightarrow\begin{pmatrix}W\\ 0
\end{pmatrix}\oplus\begin{pmatrix}Hom_H(M,\Omega^{-1}X)\\\Omega^{-1}X
\end{pmatrix}\rightarrow 0$,
\end{center}
for some $B$-module $W$, and is kept exact by the functor
$[-,\hat{G}]$. Then $pd([\begin{pmatrix}V\\ X
\end{pmatrix},\hat{G}])\leq 1$ in this case.

We next consider the case in which $V=0$, i.e., $\begin{pmatrix}V\\ X \end{pmatrix}=\begin{pmatrix}0\\
X
\end{pmatrix}$ is an $H$-module. Since $G$ is an Auslander generator
of $H -mod$, we have an exact sequence $0\rightarrow
X\longrightarrow G_0\longrightarrow G_1\rightarrow 0$ which is kept
exact when applying $Hom_H(-,G)$. Then we also get an exact sequence
of $End_\Lambda (\hat{G})$-modules

\begin{center}
$0\rightarrow [\begin{pmatrix}0\\ G_1
\end{pmatrix},\hat{G}]\longrightarrow[\begin{pmatrix}0\\ G_0
\end{pmatrix},\hat{G}]\longrightarrow [\begin{pmatrix}0\\ X
\end{pmatrix},\hat{G}]\rightarrow 0$,
\end{center}
thus showing that $pd([\begin{pmatrix}0\\ X
\end{pmatrix},\hat{G}])\leq 1$.

Finally, we consider an arbitrary indecomposable $\Lambda$-module
$\begin{pmatrix}V\\X \end{pmatrix}$. Then we have an exact
sequence

\begin{center}
$0\rightarrow \begin{pmatrix}V\\ Im(f)
\end{pmatrix}\stackrel{j}{\hookrightarrow}\begin{pmatrix}V\\ X \end{pmatrix}\stackrel{\pi}{\twoheadrightarrow}\begin{pmatrix}0\\ Coker(f) \end{pmatrix}\rightarrow
0$.
\end{center}
If now $Z$ is any indecomposable summand of $\hat{G}$, then, by
the first paragraph of this proof, the map $[j,Z]$ is surjective
except in case $Z\cong\begin{pmatrix}Kp_i\\M_i \end{pmatrix}$, for
some $i=1,...,r$. But that means that all composition factors of
the $End_\Lambda (\hat{G})$-module $Coker[j,\hat{G}]$ are of the
form $\Sigma_i=:[\begin{pmatrix}Kp_i\\M_i
\end{pmatrix},\hat{G}]/rad([\begin{pmatrix}Kp_i\\M_i \end{pmatrix},\hat{G}])$, with  $i=1,...,r$.
Suppose we prove that $pd(\Sigma_i)\leq 2$ for all $i=1,...,r$.
Then we consider the exact sequence

\begin{center}
$0\rightarrow Im[j,\hat{G}]\hookrightarrow
[\begin{pmatrix}V\\Im(f)
\end{pmatrix},\hat{G}]\longrightarrow Coker [j ,\hat{G}] \rightarrow 0$.
\end{center}
By the above paragraphs of this proof, we know that its central
term has projective dimension $\leq 1$ and, hence, we also have
$pd(Im[j,\hat{G}])\leq 1$. But then the outer nontrivial terms in
the sequence

\begin{center}
$0\rightarrow [\begin{pmatrix}0\\ Coker(f)
\end{pmatrix},\hat{G}]\longrightarrow [\begin{pmatrix}V\\X \end{pmatrix},\hat{G}]\longrightarrow Im[j,\hat{G}]\rightarrow 0$
\end{center}
have projective dimension $\leq 1$, so that $pd([\begin{pmatrix}V\\
X
\end{pmatrix}),\hat{G}]\leq 1$ and the proof would be finished.

It remains to prove that $pd (\Sigma_i)\leq 2$ for all
$i=1,...,r$. Since the canonical map
$M\otimes_BKp_i\longrightarrow M_i$ is surjective, by the first
paragraph of this proof, we know that if $Z$ is an indecomposable
summand of $\hat{G}$ such that $rad[\begin{pmatrix}Kp_i\\M_i
\end{pmatrix},Z]\neq 0$, then $Z$ is injective. We then consider the injective envelope $u:\begin{pmatrix}Kp_i\\ M_i
\end{pmatrix}\hookrightarrow\begin{pmatrix}Hom_H(M,E(M_i))\\
E(M_i)
\end{pmatrix}$, which is induced by the injective envelope $M_i\hookrightarrow E(M_i)$ in $H -mod$. Then the image of the map
$[u,\hat{G}]:[\begin{pmatrix}Hom_H(M,E(M_i))\\ E(M_i)
\end{pmatrix},\hat{G}]\longrightarrow [\begin{pmatrix}Kp_i\\ M_i
\end{pmatrix},\hat{G}]$ is precisely $rad([\begin{pmatrix}Kp_i\\M_i
\end{pmatrix},\hat{G}])$, and therefore its cokernel is
$\Sigma_i$. Note that, due to the fact that $Ext_H^1(M,M_i)=0$,
the minimal injective resolution of $\begin{pmatrix}Kp_i\\ M_i
\end{pmatrix}$ is

\begin{center}
$0\rightarrow\begin{pmatrix}Kp_i\\M_i
\end{pmatrix}\stackrel{u}{\hookrightarrow}\begin{pmatrix}Hom_H(M,E(M_i))\\
E(M_i)
\end{pmatrix}\longrightarrow \begin{pmatrix}W\\0
\end{pmatrix}\oplus\begin{pmatrix}Hom_H(M,\Omega^{-1}M_i)\\
\Omega^{-1}M_i
\end{pmatrix}\rightarrow 0$,
\end{center}
where $W$ is a $B$-submodule of $Hom_H(M,M_i)$ complementary of
$Kp_i$. We then get as projective resolution of $\Sigma_i$:

\begin{center}
$0\rightarrow [\begin{pmatrix}W\\0
\end{pmatrix}\oplus\begin{pmatrix}Hom_H(M,\Omega^{-1}M_i)\\
\Omega^{-1}M_i
\end{pmatrix},\hat{G}]\longrightarrow
[\begin{pmatrix}Hom_H(M,E(M_i))\\ E(M_i)
\end{pmatrix},\hat{G}]\longrightarrow [\begin{pmatrix}Kp_i\\ M_i
\end{pmatrix},\hat{G}]\twoheadrightarrow\Sigma_i\rightarrow 0$,
\end{center}
which shows that $pd(\Sigma_i)\leq 2$.

\end{proof}

\begin{lema} \label{finite subextension type}
Let $M$ be an $H$-module such that $Ext_H^1(M,M)=0$. The following
assertions are equivalent for an indecomposable module $U$:

\begin{enumerate}
\item $U$ belongs to $Sub(M)\cap Ker Ext_H^1(M,-)$  \item $U$ is either a direct summand of $M$ or  a
direct summand of $Ker(f)$, for some minimal right
$add(M)$-approximation $f:M'\longrightarrow X$.
\end{enumerate}
Moreover, up to isomorphism,  there only finitely many
indecomposable modules in $Sub(M)\cap Ker Ext_H^1(M,-)$.
\end{lema}
\begin{proof}
$1)\Longrightarrow 2)$ Suppose that $U$ is not a direct summand of
$M$ and let $u:U\rightarrowtail M'$ be the minimal left
$add(M)$-approximation, then the induced map
$Hom_H(M,M')\longrightarrow Hom_H(M,Coker(u))$ is surjective due to
the fact that $Ext_H^1(M,U)=0$. That means that the cokernel map
$p:M'\twoheadrightarrow Coker(u)$ is a right $add(M)$-approximation.
But $p$ is right minimal since $u$ is left minimal. Therefore $U$ is
the kernel of a minimal right $add(M)$-approximation.

$2)\Longrightarrow 1) $ If $U$ is a direct summand of $M$ there is
nothing to prove, so we assume that $U$ is not so. Let
$f:M'\longrightarrow X$ be the minimal right $add(M)$-approximation
of an indecomposable module $X$ such that $U$ is a direct summand of
$Ker(f)$. Without loss of generality, we can assume that $X\in
Fac(M)\setminus add(M)$. Since the map
$f_*:Hom_H(M,M')\longrightarrow Hom_H(M,X)$ is surjective and
$Ext_H^1(M,M')=0$, we conclude that $Ext_H^1(M,Ker(f))=0$ and hence
$U\in Sub(M)\cap Ker Ext_H^1(M,-)$.

Let $\{U_1,...,U_r\}$ be any finite set of nonisomorphic
indecomposable modules in  $Sub(M)\cap Ker Ext_H^1(M)$ including
the direct summands of $M$. Then we put $U=\oplus_{1\leq i\leq
r}U_i$ and
 claim that $Ext_H^1(U,U)=0$. Indeed there exists a monomorphism $U\rightarrowtail
 M^s$, for some $s>0$, which yields an epimorphism $0=Ext_H^1(M^s,U)\twoheadrightarrow
 Ext_H^1(U,U)$. As a consequence $U$ is a partial tilting module,
 and hence
 $r\leq n$, where $n$ is the number of simple $H$-modules.
\end{proof}

Our last auxiliary proposition leads directly to the main result.

\begin{prop} \label{Auslander generator of H}
Suppose that  that $Ext_H^1(M,M)=0$. Then there exists an
Auslander generator of $H -mod$ containing $M$ as a direct
summand.
\end{prop}
\begin{proof}
 Our goal is to construct an Auslander
generator $G$ of $H$-module containing $M$ as a direct summand.
 Let $\{V_1,...,V_m\}$ be the finite
set of indecomposable modules in $Sub(M)\cap Ker Ext_H^1(M,-)$ which
are not in $add(M)$.  We put $V=\oplus_{1\leq i\leq m}V_i$ and shall
prove that $G=:H\oplus V\oplus M\oplus DH$ is an Auslander generator
of $H -mod$.

We need to show that if $X\not\in add(G)$ is an indecomposable
$H$-module, then there is a right $add(G)$-approximation
$G_X\longrightarrow X$ whose kernel is in $add(G)$. Since $X$ is not
injective and $H$ is hereditary, every right $add(H\oplus V\oplus
M)$-approximation is already an $add(G)$-approximation. We put
$T=:tr_M(X)=\sum_{f\in Hom_H(M,X)}Im(f)$. Then the minimal right
$add(M)$-approximation $p:M'\longrightarrow X$ has $Im(p)=T$ and we
get two induced exact sequences:

\begin{center}
$0\rightarrow U\hookrightarrow
M'\stackrel{\tilde{p}}{\longrightarrow}T\rightarrow 0$

\hspace*{12cm}(*)

$0\rightarrow T\hookrightarrow
X\stackrel{q}{\longrightarrow}X/T\rightarrow 0$.
\end{center}
By Lemma \ref{finite subextension type}, we know that  $U\in
 add(M\oplus V)$. In the proof of that lemma we have shown that $M\oplus
 V$ is a partial tilting module, which implies also that $Ext_H^1(M\oplus
 V,-)$ vanishes over all modules in $Fac(M)$.
 In particular, we get that $Hom_H(V,-)$ keeps exact both sequences (*). Keeping exact
the first one means that every morphism $V\longrightarrow T$ factors
through $\tilde{p}$,  while keeping exact the second one implies
that we can choose a morphism $g:V'\longrightarrow X$ such that the
composition
$V'\stackrel{g}{\longrightarrow}X\stackrel{q}{\twoheadrightarrow}X/T$
is the minimal right $add(V)$-approximation of $X/T$.

We claim that $\begin{pmatrix}g & p \end{pmatrix}:V'\oplus
M'\longrightarrow X$ is a right $add(V\oplus M)$-approximation.
Clearly, every morphism $M\longrightarrow X$ factors through
$\begin{pmatrix}g & p \end{pmatrix}$, so we only need to see that
the same is true for every morphism $u:V\longrightarrow X$. By
definition of $g$, we have that $q\circ u$ factors through $q\circ
g:V'\longrightarrow X/T$ so that there is a $v:V\longrightarrow
V'$ such that $q\circ u=q\circ g\circ v$. Then $u-g\circ v$
factors through  $Ker(q)=T$ and we have a morphism
$w:V\longrightarrow T$ such that $j\circ w=u-g\circ v$, where
$j:T\hookrightarrow X$ is the inclusion. From the previous
paragraph we get that $w$ factors through $\tilde{p}$ and so there
is a morphism $h:V\longrightarrow M'$ such that $w=\tilde{p}\circ
h$ and then $u=p\circ h+g\circ v=\begin{pmatrix}g & p
\end{pmatrix}\circ \begin{pmatrix}v\\ h \end{pmatrix}$. This
settles our claim.

We next look at $Z=:Ker\begin{pmatrix}g & p \end{pmatrix}$. By
explicit construction of the pullback of $g$ and $p$, we see that
$Z$ fits into an exact sequence

\begin{center}
$0\rightarrow U\longrightarrow Z\longrightarrow Ker(q\circ
g)\rightarrow 0$, \hspace{1cm} (**)
\end{center}
and we already know that  $U\in add(M\oplus V)$. On the other hand,
the exact sequence

\begin{center}
$0\rightarrow
Hom_H(M,T)\stackrel{\cong}{\longrightarrow}Hom_H(M,X)\longrightarrow
Hom_H(M,X/T)\longrightarrow Ext_H^1(M,T)=0$
\end{center}
gives that $Hom_H(M,X/T)=0$ and so $Hom_H(M,Im(q\circ g))=0$. We
then get an exact sequence

\begin{center}
$0=Hom_H(M,Im(q\circ g)\longrightarrow Ext_H^1(M,Ker(q\circ
g))\longrightarrow Ext_H^1(M,V')=0$,
\end{center}
which shows that $Ker(q\circ f)\in Sub(M)\cap Ker
Ext_H^1(M,-)=add(M\oplus V)$.  But then the sequence (**) splits,
because $Ext_H^1(M\oplus V,M\oplus V)=0$. Therefore
$Z=Coker\begin{pmatrix}g & p
\end{pmatrix}\in add(M\oplus V)$.

In order to complete the desired right $add(G)$-approximation of $X$
we only need to consider a morphism $t:Q\longrightarrow X$ such that
the composition
$Q\stackrel{t}{\longrightarrow}X\stackrel{pr.}{\twoheadrightarrow}
Coker\begin{pmatrix}g & p \end{pmatrix}$ is a projective cover. It
is straightforward to see that the map

\begin{center}
$\begin{pmatrix}g & p & t \end{pmatrix}:V'\oplus M'\oplus
Q\longrightarrow X$
\end{center}
is a right $add(G)$-approximation and, by explicit construction of
the pullback of $\begin{pmatrix}g & q \end{pmatrix}$ and $t$, we
readily see that

\begin{center}
$Ker\begin{pmatrix}g & p & t \end{pmatrix}\cong Ker\begin{pmatrix}g
& p
\end{pmatrix}\oplus\Omega^1(Coker\begin{pmatrix} g &
p\end{pmatrix})$,
\end{center}
which belongs to $add(G)$ because $\Omega^1(Coker\begin{pmatrix} g
& p\end{pmatrix})$ is a projective $H$-module.
\end{proof}

As a straightforward consequence of the two propositions, we derive
the main result of the paper.

\begin{teor} \label{main theorem}
Let $H$ be a hereditary algebra of infinite representation type,
$M$ be a left $H$-module such that $Ext_H^1(M,M)=0$ and $B$ be a
semisimple subalgebra of $End_H(M)^{op}$. Then $\Lambda
=\begin{pmatrix}B & 0\\ M & H
\end{pmatrix}$  has representation dimension equal to $3$.
\end{teor}
\begin{proof}
As mentioned at the beginning of the paper, there is no loss of
generality in assuming that $H$ is basic and $B\cong
K\times\stackrel{r}{...}\times K$ is the semisimple subalgebra of
$End_H(M)^{op}$ associated to a fixed decomposition $M=\oplus_{1\leq
i\leq r}M_i$. Then from Proposition \ref{Auslander generator of H}
we know that there is an Auslander generator $G$ of $H -mod$
containing $M$ as a direct summand. Finally Proposition
\ref{Auslander generator of H versus Lambda} gives a
generator-cogenerator $\hat{G}$ of $\Lambda -mod$ such that
$gl.dim(End_\Lambda (\hat{G}))\leq 3$. Since $\Lambda$ is of
infinite representation type, we conclude that $rep.dim(\Lambda
)=3$.
\end{proof}

\end{document}